\documentclass[12pt]{amsart}
\usepackage{graphicx}
\usepackage[margin=1.3in]{geometry}
\usepackage{amsmath}
\usepackage{bbm}
\usepackage{amsthm}
\usepackage{mathtools}
\usepackage{enumitem}
\usepackage{pgfplots}
\usepackage{tikz}
\usepackage{tikz-cd}
\usetikzlibrary{shapes.geometric}
\usepgfplotslibrary{colormaps}
\usepackage{soul}
\usepackage{comment}

\newtheorem{theorem}{Theorem}[section]
\newtheorem{lemma}[theorem]{Lemma}
\newtheorem{corollary}[theorem]{Corollary}

\theoremstyle{definition}

\newtheorem{example}[theorem]{Example}

\theoremstyle{remark}

\title{Conjugacies of Expanding Skew Products on $\TT^n$}
\author{Gregory Hemenway}

\address{Dept.\ of Mathematics, The Ohio State University, Columbus, OH 43210}
\email{hemenway.math@gmail.com}

\newcommand{\leb}{\boldsymbol{m}}
\newcommand{\one}{\mathbbm{1}}

\newcommand{\NN}{\mathbbm{N}}
\newcommand{\RR}{\mathbbm{R}}
\newcommand{\sS}{\mathbbm{S}^1}
\newcommand{\ZZ}{\mathbbm{Z}}

\newcommand{\TT}{\mathbbm{T}}

\newcommand{\LL}{\mathcal{L}}

\newcommand{\MM}{\mathcal{M}}

\newcommand{\piXI}{\pi_X^{-1}}

\newcommand{\skprod}{X\times Y}

\newcommand{\holder}{\mathrm{H\ddot{o}lder}}

\pgfplotsset{compat=1.18}
\begin{document}
\date{\today}
\thanks{The author was partially supported by NSF DMS-1554794 and DMS-2154378.}
\thanks{This material is based upon work partially supported by the National Science Foundation MPS-Ascend 
Postdoctoral Research Fellowship under Grant No. DMS-2316687.}
\pagenumbering{arabic}

\begin{abstract}
{We show that any equilibrium state for a H\"older potential on the model map $\Vec{x}\mapsto d\cdot\Vec{x}\mod \ZZ^n$ on $\TT^n$ is conjugate to Lebesgue measure for an invariant expanding skew product of degree $d$. This is a generalization of a result of McMullen to higher dimensions for equilibrium states. We use an approach developed by the author using a family of nonstationary transfer operators for an expanding skew product. We also apply a Markov partition argument to classify invariant probability measures for expanding maps on $\TT^n$.}
\end{abstract}

\maketitle

\section{Introduction}

The study of conjugacies of expanding maps on the circle $\sS$ dates back to Shub \cite{Sh69} who showed that any expanding map of degree $d$ is topologically conjugate to the model mapping $E_d : \sS \rightarrow \sS$
\[ E_d(x) = d \cdot x \text{ mod }1.\]
McMullen \cite{McM} proved that there is a bijection between the Teich\"muller space of marked {topological covering maps of degree $d$} that preserve Lebesgue measure and the space of $E_d$-invariant probability measures on $\sS$, $\mathcal{M}(\sS,E_d)$. { In particular, McMullen used this result to describe the measure of maximal entropy on $\sS$ for maps including rational and expanding maps.}

The purpose of this paper is to construct conjugacies on $\TT^n=\sS\times\ldots\times\sS$ which send equilibrium states for the model map $E_{d}(\Vec{x})=d\cdot\Vec{x}\mod \ZZ^n$ to Lebesgue measure for some expanding map on $\TT^n$. {The existence of conjugacies between expanding endomorphisms on higher dimensional tori is again due to Shub (see Proposition 6 in \cite{Sh69}). Our result generalizes McMullen \cite{McM} in the case of equilibrium states {$\mu_\varphi$ for $(\TT^n,E_d)$}. In the one-dimensional case, the proofs rely on the linear ordering of the circle. We need different proofs in higher dimensions since there is no natural ordering on $\RR^2$ applicable to our purpose.}  In particular, we will prove a more general version of the following theorem (see Theorem \ref{thm:gen}).

\renewcommand{\thetheorem}{\Alph{theorem}}
\setcounter{theorem}{0}
\begin{theorem}\label{thm:A}
    Let $\mu_{\varphi}$ be an equilibrium state for $(\TT^2,E_d)$ for a H\"older potential {$\varphi$} of exponent $\alpha$. There exists a conjugacy $H\colon \TT^2\to\TT^2$ such that Lebesgue is preserved by a $C^{1+\alpha}$ expanding $F=H\circ E_d\circ H^{-1}$.
\end{theorem}

\renewcommand{\thetheorem}{\arabic{section}.\arabic{theorem}}

{We can approach the higher dimensional setting by viewing $(\TT^2,E_d)$ as a skew product. That is, for a unique equilibrium state $\mu_\varphi$ on $(\TT^2,E_d)$, we construct a Lebesgue-invariant expanding skew product $$F(x,y)=(f(x),g_x(y))$$ for which the conjugacy $H\colon \TT^2\to\TT^2$ between $F$ and $E_d$ satisfies $\mu_\varphi=H_*\text{Leb}$. This reduces the $\TT^2$ case down to a vertical collection of circular fibers $Y_x=\{x\}\times\sS$ organized by a base circle.}

{Skew products are the universal covers of random dynamical systems. In this setting the construction of equilibrium measures has been studied}
by Kifer \cite{K92}, Urbanski--Simmons \cite{SU13}, Stadlebuar--Suzuki--Varandas \cite{SSV}, and Hafouta \cite{H20,Ha23}.
They all use a nonstationary transfer operator
\[
\LL_x\colon C(Y_x)\to C(Y_{fx}), \quad
\psi_x(y)\mapsto\sum_{\overline{y}\in g_x^{-1}y}e^{\varphi_x(\overline{y})}\psi_x(\overline{y})
\]
acting on continuous functions which are nonvanishing along fibers. Denker--Gordin \cite{DG99} {used these transfer operators to study} Gibbs measures for fibred systems by showing that there is a family of conditional measures $\{\mu_x\}$ and a continuous function $\Phi\colon \sS\to\RR$ on the base such that for any $\psi\in C(\TT^2)$
\[\int_E\LL_x\psi_x{d\mu_{fx}}=\int_E\Phi(x)e^{-\varphi_x}d\mu_x\]
whenever $F|_E$ is invertible.

In \cite{H23}, the current author used this family of nonstationary operators to construct such a family of conditional measures for equilibrium states on non-uniformly expanding skew products{; i.e. for any $\psi\in C(X\times Y)$
\[
\int\psi d\mu_\varphi=\int\mu_x(\psi)d\hat{\mu}(x)
\]
where $\hat{\mu}$ is the projection of the equilibrium state $\mu_\varphi$ onto the base $\sS$. This allows us to treat each fiber as its own circle. We can then essentially apply McMullen's result to each fiber and glue them back together with another application of McMullen in the base. This allows us to build the desired conjugacy. We then use the eigendata of the transfer operators to establish the regularity of $F$.}
We review the necessary details in Section \ref{sec:skprod}.

We also show that the number of such orientation preserving conjugacies $H$ is equal to the degree of the maps. In \cite{McM}, McMullen establishes his bijection using {ideas from} Teich\"muller theory. Thus, there is an equivalence class $\{(f,\phi)\}$ of expanding map, conjugacy pairs determined by actions of Moduli space. We avoid any uses of Teich\"muller spaces. Instead, our approach involves use of Markov partitions to count the number of desired conjugacies. It is well-known that any expanding circle map of degree $d$ is semi-conjugate to a shift $\sigma$ on $d$ symbols $\Sigma_d=\{0,\ldots,d-1^\NN\}$. We show that the number of conjugacies is equivalent to permutations of the canonical Markov partition. See Corollary \ref{cor:} for details.

During the latter stages of this manuscript's preparation, an inquiry was raised concerning the regularity of the conjugacy. McMullen's result gives $C^0$ regularity for the conjugacy $\phi$. It is well-known that this can be easily upgraded to H\"older regularity. However, we cannot expect higher regularity without information about how periodic data is transformed by the conjugacy. For $f,g\in C^k\ (k\geq 2)$ expanding maps on $\sS$ that are topologically conjugated by $\phi$, if the Lyapunov exponents at corresponding periodic orbits are the same, then $\phi\in C^{k-1}$ (see Shub and Sullivan \cite{SS85}). We refer the reader to work by de la Llave \cite{dlL92} and Gogolev and Rodriguez-Hertz \cite{GRH21} for more on the rigidity of higher dimensional expanding maps.\\


The outline of this paper is as follows. In Section \ref{sec:setting}, we describe our setting in detail and define some important concepts. In Section \ref{sec:skprod}, we define the fiberwise transfer operators and state some results by the author that will be important for our analysis here. In Section \ref{sec:2torus}, we prove Theorem A using the nonstationary techniques described in Section 3. We also show that the resultant skew product $F\colon \TT^2\to\TT^2$ is $C^{1+\alpha}$. In Section \ref{sec:ntorus}, we describe an iterative process to extend the results of Section \ref{sec:2torus} to $\TT^n$ for any $n\in\NN$.

\subsection*{Acknowledgments}{The author thanks Vaughn Climenhaga for plenty of insightful discussions which motivated this project. The author also thanks James Marshall Reber for helpful discussions that contributed to the Markov partition arguments for the cardinality of the desired conjugacies.}

\section{Background}
\label{sec:setting}

\subsection{Existence and Uniqueness of Equilibrium States}
\label{sec:ESexun}

{In this section, we review the thermodynamic formalism of dynamical systems which was initiated in the seventies by the work of those like Bowen and Kolmogorov. For a thorough discussion of the entropy and more generally pressure of a continuous dynamical system, see  Walters \cite{WaltersET} or Peterson \cite{P83}.

Let $X$ be a compact metric space and $f\colon X\to X$ be continuous. The topological pressure of a potential $\varphi\colon X\to\RR$ is a weighted growth rate over $(n,\epsilon)$-separated sets. We wish to choose a distinguished $f$-invariant Borel probability measure $\mu\in \MM(X,f)$. The variational principle gives the following relationship between a system's topological pressure $P(\varphi)$ and free energy of its invariant measures. 

\begin{theorem}[Variational Principle]
    Given a continuous potential $\varphi\colon X\to \RR$,
    \[P(\varphi)=\sup\Big\{h_\nu(f)+\int\varphi\ d\nu \colon \nu\in\mathcal{M}(X,f) \Big\}\]
    where $h_\nu(f)$ denotes the measure-theoretic entropy of $f$.
\end{theorem}

An \emph{equilibrium state} $\mu\in \mathcal{M}(X,f)$ for $\varphi$ is an invariant measure that achieves the supremum in the variational principle. These measures are a generalization of the measure of maximal entropy which is the equilibrium state for the potential $\varphi\equiv 0$.

We say that a potential $\varphi\colon X\to\RR$ is \emph{$\alpha$-H\"older continuous} for $\alpha>0$ if \[|\varphi|_\alpha:=\sup_{x\not=x'}\frac{|\varphi(x)-\varphi(x')|}{d(x,x')^\alpha}<\infty.\] We denote by $ C^\alpha(X)$ the Banach space of $\alpha$-$\holder$ continuous functions on $X$. Walters \cite{W78} proved the existence and uniqueness of equilibrium states for H\"older potentials on expanding maps using the Ruelle-Perron-Frobenius transfer operator acting on $ C^\alpha(X)$ via
\[
\LL_\varphi\psi(x)=\sum_{\overline{x}\in f^{-1}x}e^{\varphi(\overline{x})}\psi(\overline{x}).
\]
}

\begin{theorem}[RPF Theorem (See \cite{B08} \& \cite{W78})]
\label{thm:RPF}
Let $X$ be a compact, connected manifold and $f\colon X\to X$ is uniformly expanding. For any H\"older $\Phi\colon X\to\RR$, the following hold:
 \begin{enumerate}
\item \label{meas} There is a unique probability measure ${\nu}\in \MM(X)$ with the property that $\LL_\Phi^*{\nu}$ is a scalar multiple of ${\nu}$.
\item \label{func} There is a unique positive continuous function ${h}\in C^\alpha(X)$  with the property that $\LL_\Phi{h}$ is a scalar multiple of ${h}$ and $\int_X{h}d{\nu}=1$.
\item \label{value} The eigenvalues $\lambda$ associated to ${\nu}$ and ${h}$ are the same.
\item The unique equilibrium state for $\Phi$ is ${\mu}={h}{\nu}$.
 \end{enumerate}
\end{theorem}

{
We note that $\lambda=e^{P(\varphi)}$. Several times throughout this paper we will need to normalize our potential $\varphi$. We note that a potential $\psi$ is cohomologous (up to a constant) to $\varphi$ if there exists a $u\colon X\to\RR$ such that $\varphi-\psi=u-u\circ f+c$ for some $c\in\RR$. Bowen \cite{B08} shows that two potentials are cohomologous up to a constant if and only if they produce the same equilibrium states. Thus, if we replace $\varphi$ with 
\begin{equation}
    \label{eqn:cohomologuous}
    \Tilde{\varphi}=\varphi+\log h-\log h\circ f-\log\lambda,
\end{equation}
then $\LL_{\Tilde{\varphi}}\one=\one$ and the equilibrium state $\mu$ satisfies $\LL_{\Tilde{\varphi}}^*\mu=\mu$.}

{ It is worth mentioning that works in Teich\"muller theory like McMullen \cite{McM08} and He et al \cite{He25} use thermodynamic formalism to categorize metrics on hyperbolic surfaces using the pressure norm and Weil-Petersson metric on $C(X)$.}

\subsection{Semiconjugacies of Expanding Maps on $\sS$}
\label{ssec:McM}

Throughout this paper, we will denote the Lebesgue measure on $\sS$ by $\leb$. Let $\text{Exp}_d(\sS)$ denote the space of topological covering maps $f\colon \sS\to \sS$ of degree $d>1$ that preserve Lebesgue measure. Note that such maps are expanding maps; i.e. there exists a constant \( \lambda > 1 \) such that
\[
d(f(x), f(y)) \geq \lambda d(x, y) \quad \text{for all } x, y \in X.
\]
Shub showed that every $f\in \text{Exp}_d(\sS)$ is topologically conjugate to the model mapping $E_d(t)=d\cdot t\mod 1$ (see \cite{Sh69}).

\begin{figure}[ht!]
    \centering
    \[
    \begin{tikzcd}
        \sS \ar{r}{E_d} \arrow[d,"\phi"']
        & \sS \arrow[d, "\phi"]\\
        \sS \arrow[r, "f"']
        & \sS
    \end{tikzcd}
    \]
    \caption{Commutative diagram between $(\sS,E_d)$ and $(\sS,f)$.}
\end{figure}

McMullen proves the following.

\begin{theorem}
\textup{(McMullen \cite[Corollary 2.2 \& Theorem 1.1]{McM})}
\label{thm:McM}
    For any $f\in{\text{Exp}}_d(\sS)$ with $ d>1$, there is a $\phi\in \overline{\text{Exp}}_1(\sS)$, {unique up to the actions by automorphisms of $E_d$}, such that $\phi\circ f=E_d\circ \phi$ in $\overline{\text{Exp}}_d(\sS)$.
    Moreover, $\phi$ determines a unique invariant probability measure $\nu_{(f,\phi)}=\phi_*\leb$.
\end{theorem}

We provide an alternate proof that avoids the use of Teich\"muller theory. Instead, we appeal to the Markov structure of expanding maps to  count the conjugacies that satisfy Theorem A. We say a finite collection $R=\{R_0,\ldots,R_{\ell-1}\}$ of subsets of $\sS$ form a \emph{Markov partition} for $f$ if
\begin{enumerate}
    \item The collection $R$ covers $\sS$; i.e. $\sS=\bigcup_{i=0}^{\ell-1}R_i$.
    \item The interiors of the $R_i$'s are pairwise disjoint; i.e. $\text{int}R_i\cap\text{int}R_j=\emptyset$.
    \item For any $i$, $R_i=\overline{\text{int}R_i}$
    \item (the Markov property): for all $i,j$ each point of $\text{int}(R_j)$ has the same number of inverse images in $R_i$ under $f$.
\end{enumerate}
The $R_i$'s are called the \emph{rectangles} for the partition. For more on Markov partitions for expanding circle maps, see \cite{S91}.

It is well-known that $(\sS,E_d)$ is semi-conjugacies to the full shift on $d$ symbols. We can define a canonical Markov partition 
\[ A_k := \left[\frac{k-1}{d}, \frac{k}{d} \right]\]
which gives a $(d\text{ to }1)$ coding of $E_d$ in $\{0,\ldots,d-1\}^\NN$; i.e.  
$$\sigma^n\omega_0\omega_1\ldots=\omega_1\omega_2\ldots \Longleftrightarrow  f^nx\in A_{\omega_n}.$$
The following corollary shows that topological conjugacies are determined by labelings of Markov partitions. Also that these partitions define invariant probability measures on $(\sS,E_d)$.

{
\begin{corollary}\label{cor:}
    Let $f\in Exp_d(\sS)$. Given an invariant probability measure $\nu_f$, there are $d$ orientation preserving and $d$ orientation reversing conjugacies to expanding $(\sS,f)$ for which $\nu=\phi_*\nu_f\in\mathcal{M}(\sS,E_d)$.
\end{corollary}

\begin{proof}
    Let $\{A_k\}$ be the canonical Markov partition for $E_d$. This partition induces a probability vector $p_\nu=(\nu(A_0),\ldots,\nu(A_{d-1}))$ for the measure $\nu$. 
    
    Let $\nu_f$ be some reference measure on $(\sS,f)$. Choose a conjugacy $\phi\colon \sS\to\sS$ for which $E_d\circ \phi=\phi\circ f$. The map $f\in Exp_d(\sS)$ has a Markov partition $\{B_k\}$. Topological conjugacies between $E_d$ and $f$ correspond to scalings and permutations on the $d$ elements of $A_k$. That is, the homeomorphism $\phi$ induces a map $\rho \in \text{Sym}(d)$ such that $B_k = \phi(A_{\rho(k)})$. Moreover,
    \begin{align*}
        \rho(p_\nu)
        &=(\gamma_0\phi^{-1}_*\nu(A_0)
        ,\ldots,
        \gamma_{d-1}\phi^{-1}_*\nu(A_{d-1}))\\
        &=(\gamma_0\nu(B_0)
        ,\ldots,
        \gamma_{d-1}\nu(B_{d-1}))
    \end{align*}
    for some $(\gamma_0,\ldots,\gamma_{d-1})$ with $\gamma_j\geq 0$ for all $0\leq j\leq d-l$. Therefore, different solutions to $\gamma_j\nu(B_j)=\nu_f(B_j)$ for each $j$, define different probability vectors and thus probability measures on $(\sS,f)$. 
    
    Note that it is not true that all $\rho \in \text{Sym}(d)$ can be realized via a topological conjugacy $\phi$. Indeed, $\phi$ must preserve the cyclic ordering of the intervals. Thus, conjugacies are uniquely determined by the first choice of symbols in the Markov partition and whether we rotate this partition clockwise or counter-clockwise. Hence, there are $d$ orientation preserving and $d$ orientation reversing such $\rho$. Hence, there are $2d$ conjugacies to Lebesgue preserving expanding maps.
\end{proof}

We note that all of these conjugacies would belong to the same equivalence class in the setting of \cite{McM}. That is, any two of these conjugacies must be rotations of each other. Thus, the following corollary agrees with McMullen's bijection. In particular, given an equilibrium state $\mu$ for $(\sS,E_d,\varphi)$, Corollary \ref{cor:} shows that there are $2d$ conjugacies that send $\mu$ to $\leb$. In Sections \ref{sec:2torus} and \ref{sec:ntorus}, we generalize this argument to count the number of such conjugacies on higher dimensional tori.}

\subsection{Statements of Main Results}

A map $F\colon \TT^n\to\TT^n$ is a \emph{linear expanding endomorphism} if there is an $n\times n$ integer matrix $A$ whose eigenvalues all have absolute value greater than one such that { $F(\Vec{x})=A(\Vec{x}) \mod \ZZ^n$. The model map $E_d$ is the linear endomorphism which corresponds to the diagonal matrix with all positive entries equal to $d$.} Shub \cite[Proposition 6]{Sh69} showed that every expanding $F\colon \TT^n\to\TT^n$ is topologically conjugate to a linear expanding endomorphism. {We equip} $(\TT^n,E_{d})$ with a H\"older potential $\varphi\colon \TT^n\to\RR$. By Theorem \ref{thm:RPF}, there exists a unique equilibrium state $\mu$ for $\varphi$ on $(\TT^n,E_d)$. Our main result is as follows.

\begin{theorem}\label{thm:gen}
    Fix $n\geq 2$ and consider $E_d\colon \TT^n\to\TT^n$. Let $\varphi\colon\TT^n\to\RR$ be a H\"older continuous function and $\mu$ its corresponding equilibrium state. Then there exists a conjugacy $H\colon\TT^n\to\TT^n$ to a $C^{1+\alpha}$ expanding skew product $F$ that transports the equilibrium state $\mu$ to Lebesgue measure. Moreover, there are only $d$ orientation preserving conjugacies that satisfy this.
\end{theorem}

\section{Skew Products}
\label{sec:skprod}

Let $X$ and $Y$ be compact, connected manifolds. We will refer to $X$ as the base and $\{Y_x:=\{x\}\times Y\}_{ x\in X}$ as the fibers of the product. Note that each fiber $Y_x$ can be identified with $Y$. We will make the necessary distinctions as needed. Denote by $d$ the $L^1$ distance on $\skprod$ and by $\pi_X$ and $\pi_Y$ the natural projection maps from $\skprod$ onto $X$ and $Y_x$, respectively. 

Let $F$ be a continuous skew product on $X\times Y$; i.e. there are continuous maps $f\colon X\to X$ and $\{g_x\colon Y_x\to Y_{fx}|\ x\in X\}$ such that \[F(x,y)=(f(x),g_x(y)). \] 
To understand the dynamics of $F$ on $\TT^2$, define a family of maps $$g_x^k:=\ g_{f^{k-1}x}\circ\dots\circ g_x\colon Y_x\to Y_{f^nx}$$ where $k\geq 0$ and $x\in X$.

For each $k\geq 0$, define the $k^{th}$-Bowen metric as \[d_k((x,y),(x',y'))=\max_{0\leq i\leq k}\{d(F^i(x,y),F^i(x',y'))\}.\] Denote the $k^{th}$-Bowen ball centered at $(x,y)$ of radius $\delta>0$ by $$B_k((x,y),\delta)=\{(x',y')\colon d_k((x,y),(x',y'))<\delta\}.$$

\subsection{Prior Results}

Observe that we can decompose the product transfer operator for skew products into \[\LL_\varphi\psi(x,y)=\sum_{(\overline{x},\overline{y})\in F^{-1}(x,y)}e^{\varphi(\overline{x},\overline{y})}\psi(\overline{x},\overline{y})=\sum_{\overline{x}\in f^{-1}x}\sum_{\overline{y}\in g_{\overline{x}}^{-1}y}e^{\varphi(\overline{x},\overline{y})}\psi(\overline{x},\overline{y}).\] This gives rise to a fiberwise transfer operator on the fibers of $\skprod$.

We disintegrate $\varphi$ into a family of fiberwise potentials $\{\varphi_x(\cdot)=\varphi(x,\cdot)\}_{x\in X}$. For every $x\in X$, let $\LL_x\colon C(Y_x)\to C(Y_{f(x)})$ be defined by \[
    \LL_x\psi_x(y)=\sum_{\overline{y}\in g_x^{-1}y}e^{\varphi_x(\overline{y})}\psi_x(\overline{y})
\] for any $\psi\in  C(\skprod)$. We are interested in iterating this fiberwise operator. Thus, defin \[\LL_x^k=\LL_{f^{k-1}x}\circ\cdots\circ \LL_x\colon C(Y_x)\to C(Y_{f^nx}).\] Along with each of these fiberwise operators, we define their corresponding duals $\LL_x^*$ by sending a probability measure $\eta\in \mathcal{M}(Y_{f(x)})$ to the measure $\LL_x^*\eta\in \mathcal{M}(Y_x)$ defined by the relationship that for any $\psi\in  C(\skprod)$, one has \[\int \psi\,d(\LL_x^*\eta)=\int \LL_x\psi\,d\eta.\]

In \cite[Theorem A]{H23}, the current author built on ideas of Denker and Gordin \cite{DG99} in order to prove Theorem \ref{thm:conds} below. The motivation was to produce conditional measures for equilibrium states for the doubling map on the 2-torus, $\mathbbm{T}^2$, which allows us to approach the two dimensional case one dimension at a time. 

\begin{theorem}[Hemenway \cite{H23}]\label{thm:conds}
Let $X$ and $Y$ be compact connected Riemannian manifolds and $(\skprod,F)$ be a Lipschitz expanding skew product. Let $\varphi$ be {almost constant} H\"older continuous potential on $\skprod$ and $\mu$ be its corresponding equilibrium state. Then the following are true. \begin{enumerate}
    \item \label{itm:pot}The potential $\Phi(x)=\lim_{k\to\infty}\log{\LL_x^{k+1}\mathbbm{1}(y)}\big/{\LL_{fx}^{k}\mathbbm{1}(y)}$ exists (independent of choice of $y\in Y$), is H\"older continuous, and satisfies $P(\varphi)=P(\Phi)$.
    \item \label{itm:baseES}The unique equilibrium state for $\Phi$ is $\hat{\mu}=\mu\circ\piXI$.
    \item  There is a unique family of measures $\{\nu_x\colon x\in X\}$ such that $\nu_x(Y_x)=1$ and $$\LL_x^*\nu_{f(x)}=e^{\Phi(x)}\nu_x.$$
    \item The map $x\mapsto\nu_x$ is weak$^*$-continuous.
    \item \label{itm:conds} Let $\hat{h}$ and $\hat{\nu}$ be the eigendata of $\LL_\Phi$, i.e. $\LL_\Phi^*\hat{\nu}=e^{P(\Phi)}\hat{\nu}$, $\LL_\Phi\hat{h}=e^{P(\Phi)}\hat{h}$, and $\int\hat{h}d\hat{\nu}=1$. Then the measures $\mu_x={h(x,\cdot)}/{\hat{h}(x)}\nu_x$ are probability measures on $Y_x$ such that \[\mu=\int_X\mu_x\ d\hat{\mu}(x).\]
\end{enumerate} 
\label{sec:ES}
\end{theorem}

\section{Conjugacies on $\TT^2$}
\label{sec:2torus}

Let $\varphi\colon\TT^2\to\RR$ be H\"older on $(\TT^2,E_d)$ and consider its unique equilibrium state $\mu$. 
Consider the transverse measure $\hat{\mu}=\mu\circ\pi_X^{-1}$ and the conditional measures $\{\mu_x\}$ given by \cite{H23}. Using McMullen \cite{McM}, we get a conjugacy $\hat{\phi}\colon \sS\to\sS$ to a uniformly expanding system $(\sS,f)$ for which $\hat{\phi}$ pushes Lebesgue $\leb$; i.e.
\begin{equation}
\label{eqn:baseConj}
    f\circ \hat{\phi} = \hat{\phi} \circ E_d
    \quad\text{ and }\quad
    \hat{\mu}=\hat{\phi}_*\leb.
\end{equation}

{
Recall from equation \eqref{eqn:cohomologuous} that we can normalize the potential $\Phi\colon \sS\to \RR$ from item \eqref{itm:pot} of Theorem \ref{thm:conds}. The following lemma shows that this normalization $\Tilde{\Phi}$ defined the derivative of $f$. Therefore, the regularity of $\Tilde{\Phi}$ implies that $f$ is $C^{1+\alpha}$.
}

\begin{lemma}\label{lem:transDeriv}
    If $f$ is the uniformly expanding map given by equation \eqref{eqn:baseConj}, then {$f$ has continuous derivative} $f'(x)=e^{-\Tilde{\Phi}(\hat{\phi}^{-1}({x}))}$.
\end{lemma}

\begin{proof}
Fix $x\in \sS$ and let $\delta>0$. Note that 
\[f(x+\delta)-f(x)=\leb\bigg(\big[f(x),f(x+\delta)\big)\bigg)=\leb\bigg(f\circ\big[\hat{\phi}(\overline{x}),\hat{\phi}(\overline{x}+\overline{\delta})\big)\bigg)=\hat{\mu}\bigg(E_d\big([\overline{x},\overline{x}+\overline{\delta})\big)\bigg)  \] where $\overline{\delta}>0$ is given by the continuity of $\hat{\phi}$. Also, \[\delta=\leb\bigg(\big[{x},{x}+{\delta}\big)\bigg)=\leb\bigg(\big[\hat{\phi}(\overline{x}),\hat{\phi}(\overline{x}+\overline{\delta})\big)\bigg)=\hat{\mu}\bigg(\big[\overline{x},\overline{x}+\overline{\delta}\big)\bigg).\]
For convenience, we let $I:=[\overline{x},\overline{x}+\overline{\delta})$. Then 
\begin{align*}
        f'(x)&=\lim_{\delta\to 0}\frac{f(x+\delta)-f(x)}{\delta}=\lim_{\delta\to 0}\frac{\hat{\mu}\bigg(E_d\big([\overline{x},\overline{x}+\overline{\delta})\big)\bigg)}{\hat{\mu}\bigg(\big[\overline{x},\overline{x}+\overline{\delta}\big)\bigg)}=\frac{\displaystyle \int\mathbbm{1}_{E_d(I)}\hat{h}(z)d\hat{\nu}(z)}{\displaystyle 
 \int\mathbbm{1}_I(z)\hat{h}(z)d\hat{\nu}(z)}
\end{align*}

{
Items \eqref{itm:pot} and \eqref{itm:baseES} in Theorem \ref{thm:conds} implies that there is an eigenmeasure $\hat{\nu}\in\mathcal{M}(X)$ such that $\LL_\Phi^*\hat{\nu}=e^{P(\Phi)}\hat{\nu}$. So \begin{align*}
    \int_I \hat{h}(z)d\hat{\nu}(z)
    &=e^{-P(\Phi)}\int\LL_\Phi(\mathbbm{1}_I\hat{h})(z)d\hat{\nu}(z)\\
    &=e^{-P(\Phi)}\int\sum_{z'\in E_d^{-1}z\cap I} e^{\Phi(z')}\hat{h}(z')d\hat{\nu}(z)\\
    &=e^{-P(\Phi)}\int_{E_d(I)} e^{\Phi(E_d|_{I}^{-1}z)}\hat{h}(E_d|_{I}^{-1}z)d\hat{\nu}(z)
\end{align*}
where the last equality holds since $E_d$ expanding implies $E_d^{-1}z\cap I$ is a singleton.
}

Let $\varepsilon>0$. Note that since $\hat{h}$ and $\Phi$ are continuous, there is a $\delta_0$ such that if $\delta<\delta_0$, we have
\[e^{-\varepsilon}\leq\frac{\hat{h}(z)}{\hat{h}(E_d(\overline{x}))} \leq e^{\varepsilon},\quad |\Phi(E_d|_I^{-1}z)-\Phi(\overline{x})|\leq \varepsilon,\quad \text{and}\ e^{-\varepsilon}\leq\frac{\hat{h}(E_d|_I^{-1}z)}{\hat{h}(\overline{x})}\leq e^{\varepsilon}.\]
Therefore,
\[f'(x)=\lim_{\delta\to 0}\frac{\displaystyle \int_{E_d(I)}\hat{h}(z)d\hat{\nu}(z)}{\displaystyle \int_I\hat{h}(z)d\hat{\nu}(z)}\leq\frac{\hat{h}(E_d(\overline{x}))e^\varepsilon}{\hat{h}(\overline{x})e^{-P(\Phi)}e^{\Phi(\overline{x})}e^{-2\varepsilon}}=\frac{e^{P(\Phi)}\hat{h}(E_d(\overline{x}))}{\hat{h}(\overline{x})e^{\Phi(\overline{x})}}e^{3\varepsilon}.\] A similar inequality shows that \[\frac{e^{P(\Phi)}\hat{h}(E_d(\overline{x}))}{\hat{h}(\overline{x})e^{\Phi(\overline{x})}}e^{-3\varepsilon}\leq f'(x) \leq \frac{e^{P(\Phi)}\hat{h}(E_d(\overline{x}))}{\hat{h}(\overline{x})e^{\Phi(\overline{x})}}e^{3\varepsilon}.\] This holds for all $\varepsilon>0$ so we get $f'(x)=e^{-\Tilde{\Phi}(\hat{h}(\overline{x}))}$.
\end{proof}

\subsection{Construction of conjugacy on $\TT^2$}

In this section, we use Theorem \ref{thm:conds} to build the conjugacy on $\TT^2$. 

Let $f$ be given by equation \eqref{eqn:baseConj}. For any $x\in\sS$, let $F_x\colon\{x\}\times\sS\to\{E_{d}(x)\}\times\sS$ be the $\times d$ map in the second coordinate. Then {$E_d(x,y)=(dx,F_x(y))\mod \TT^2$}. Let $H_X\colon \TT^2\to\TT^2$ be defined by $H_X(x,y)=(\hat{\phi}(x),y)$. Then $\Tilde{F}(x,y)=(f(x),F_x(y))$ satisfies $\Tilde{F}\circ H_X=H_X\circ E_d$ and $\Tilde{\mu}=\int\mu_xd\leb(x)$. This conjugacy organizes the fibers based on the new nonlinear map $f$.

\begin{figure}[h!]
    \centering
    \[\begin{tikzcd}
        \{x\}\times\sS \arrow[r, "F_x"] \arrow[d,"\phi_x"']
        & \{E_{d}(x)\}\times\sS \arrow[d, "\phi_{f(x)}"]\\
        \{x\}\times \sS \arrow[r, "g_x"']
        & \{f(x)\}\times \sS
    \end{tikzcd}\]

    \caption{Commutative diagram between $(Y_x,F_x)$ and $(Y_x,g_x)$.}
    \label{fig:fibers}
\end{figure}

Now we apply the same idea to translate $\Tilde{F}$ to a new nonlinear skew product $F\colon\TT^2\to\TT^2$. Analogous techniques to those above define a family of coordinate changes $\{\phi_x\colon Y_x\to Y_{x}\}$ which sends $(\sS,E_d,\mu_x)$ to a uniformly expanding map $g_x\colon \{x\}\times \sS\to\{f(x)\}\times\sS$ such that $\leb(A)=\mu_x(\phi_xA)$. {For each fiber, define $\phi_x(y)=\mu_x([0,y])$.} Note that $\phi_x(0)=0$ and $\phi_x(1)=1$. Thus, $\phi_x$ is a homeomorphism since $\mu_x$ is continuous implies $\phi_x$ is bijective. Note that as in Figure \ref{fig:fibers}, the nonlinear maps $g_x$ all satisfy
\begin{equation}
    \label{eqn:fibConj}
    g_x=\phi_{f(x)}\circ F_x\circ \phi_x^{-1}.
\end{equation}

Now let $H_Y\colon \TT^2\to\TT^2$ be defined as $H_Y(x,y)=(x,\phi_x(y))$. Then the skew product $F(x,y)=(f(x),g_x(y))$ satisfies $F\circ H_Y=H_Y\circ \Tilde{F}$ and {$\leb=\Tilde{\mu}\circ H_Y$}. 
{
We now show that these conjugacies can be combined into a conjugacy between $(\TT^2,E_d)$ and $(\TT^2,F)$.
}

\begin{figure}[h!]
    \centering
    \begin{minipage}{0.25\textwidth}
        \centering
    \begin{tikzpicture}[scale=0.4]
    \begin{axis}[xmin = 0, xmax = 1,
    ymin = 0, ymax = 1, xlabel={$x$}, ylabel={$y$}]
    \addplot [] table {
        0  0
        0.609  0.711
        };
    \addplot [dotted] table {
        0.609  0
        0.609  0.711
        };
    \end{axis}
    \end{tikzpicture}
\end{minipage}$\overset{H_X}{\xrightarrow{\hspace*{5mm}}}$\begin{minipage}{0.25\textwidth}
        \centering
    \begin{tikzpicture}[scale=0.4]
    \begin{axis}[xmin = 0, xmax = 1,
    ymin = 0, ymax = 1, xlabel={$\overline{x}$},
    ylabel={$\overline{y}$}]
    \addplot [] table {
        0  0
        0.218  0.711
        };
    \addplot [dotted] table {
        0.218  0
        0.218  0.711
        };
    \end{axis}
    \end{tikzpicture}
\end{minipage}
$\overset{H_Y}{\xrightarrow{\hspace*{5mm}}}$\begin{minipage}{0.25\textwidth}
        \centering
    \begin{tikzpicture}[scale=0.4]
    \begin{axis}[xmin = 0, xmax = 1,
    ymin = 0, ymax = 1, xlabel={$\overline{x}$},
    ylabel={$\overline{y}$}]
    \addplot [] table {
        0  0
        0.218  0.422
        };
    \addplot [dotted] table {
        0.218  0
        0.218  0.422
        };
    \end{axis}
    \end{tikzpicture}
\end{minipage}
    
    \caption{The homeomorphisms $H_X$ and $H_Y$ reorganizing fibers according to the dynamics of $F$.}
\end{figure}
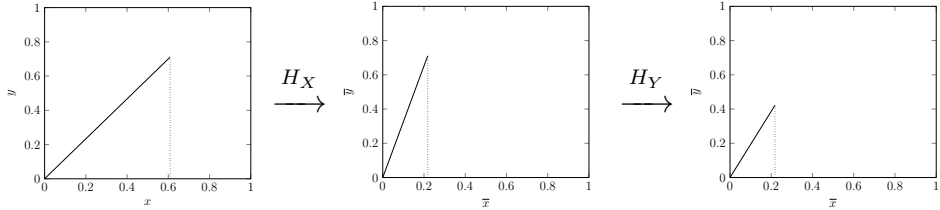

\begin{theorem}
\label{thm:torusConj}
    There is a conjugacy $H\colon\TT^2\to\TT^2$ that takes $(\TT^2,E_d,\mu)$ to $(\TT^2,F,\leb)$.
\end{theorem}

\begin{proof}
    Let $H_X$ and $ H_Y$ be defined as above. Suppose $H=H_Y\circ H_X$. By \cite{H23}, $\phi_x$ is continuous in $x$. Thus, $H$ is a homeomorphism as the composition of a continuous collection of fiberwise homeomorphisms.
    
    The homeomorphism $H$ is a conjugacy on $\TT^2$ since 
    \begin{align*}
        F\circ H(x,y)&=F(\hat{\phi}(x),\phi_{\hat{\phi}(x)}(y))\\
        &=(f(\hat{\phi}(x)),g_{\hat{\phi}(x)}(\phi_{\hat{\phi}(x)}(y)))\\
        &=(\hat{\phi}(dx)),\phi_{\hat{\phi}(dx)}(dy)))\\
        &=H\circ E_d(x,y).
    \end{align*}
    Note that we used equation \eqref{eqn:baseConj} and \eqref{eqn:fibConj} in the third equality.

    Also, for any $\psi\in C(\TT^2)$, we have 
    \begin{align*}
        \int \psi\circ H^{-1}(x,y) d \leb^2(x,y)&=\int\int\psi\circ (H_Y\circ H_X)^{-1}(x,y) d \leb(y)d \leb(x)\\
        &=\int\int\psi\circ H_X^{-1}(x,y) d\mu_x(y)d{\leb}(x)\\
        &=\int\int\psi(x,y) d\mu_x(y)d\hat{\mu}(x)\\
        &=\int\psi d\mu.\qedhere
    \end{align*}
\end{proof}

{
An analogous proof to Lemma \ref{lem:transDeriv} gives the following.

\begin{lemma}\label{lem:fibDeriv}
    Let $g_x$ be the $\leb$-invariant nonlinear coordinate change of the doubling map $F_x$ on $\{x\}\times\sS$ that sends $\mu_x$ to $\leb$. We have $g_x'(y)=e^{-\Tilde{\varphi}_x(\phi_x^{-1}({y}))}$
\end{lemma}
}

\begin{theorem}
\label{thm:det}
Let $F$ and $H$ be as in Theorem \ref{thm:torusConj}. Then $\varphi\circ H^{-1}$ is the log-Jacobian for $F$; i.e. $|\det DF(x,y)|=e^{-\varphi(H^{-1}(x,y))}$
\end{theorem}

\begin{proof}
    Note that the differential matrix for $F$ is
    \[
    DF(x,y)=\Bigg[\begin{matrix}f'(x)&0\\ \frac{\partial g_x}{\partial x}(x,y)&g_x'(y)\end{matrix}\Bigg].
    \]
    Lemmas \ref{lem:transDeriv} and \ref{lem:fibDeriv} imply $f'(x)$ and $g_x'(y)$ exist. Note that Corollary 3.9 in \cite{H23} and item \eqref{itm:conds} from Theorem \ref{thm:conds} implies that the Radon-Nikodym derivative along $\{\mu_x\}$ in continuous in $x$. {In particular, this implies that $\partial g_x/\partial x$ exists and is H\"older continuous.}
    
    Since $F$ is a skew product, $\partial f/\partial y=0$ which implies that $\partial g_x/\partial x$ does not contribute to the determinant.
    By Lemmas \ref{lem:transDeriv} and \ref{lem:fibDeriv} and Theorem \ref{thm:conds}, we have \begin{align*}
        \det DF(x,y)
        &=f'(x)g_x'(y)\\
        &=e^{-\Tilde{\Phi}(\hat{\phi}(\overline{x}))}e^{-\Tilde{\varphi}_{\hat{\phi}(\overline{x})}(\phi_x(\overline{y}))}\\
        &=\frac{e^{P(\Phi)}\hat{h}(E_d(\overline{x}))}{\hat{h}(\overline{x})e^{\Phi(\overline{x})}}\cdot\frac{e^{\Phi(x)}h_{E_d(\overline{x})} (F_x(\overline{y}))}{h_{\overline{x}} (\overline{y})e^{\varphi_{\overline{x}}(\overline{y})}}\\
        &=\frac{e^{P(\Phi)}\hat{h}(E_d(\overline{x}))h_{E_d(\overline{x})} (F_x(\overline{y}))}{\hat{h}(\overline{x})h_x (\overline{y})e^{\varphi_{\overline{x}}(\overline{y})}}\\
        &=\frac{e^{P(\varphi)}h (E_d(\overline{x}),F_x(\overline{y}))}{h(\overline{x},\overline{y})e^{\varphi({\overline{x}},\overline{y})}}=e^{-\Tilde{\varphi}(\overline{x},\overline{y})}. \qedhere
    \end{align*}
\end{proof}

{In dimension $n=2$, Gogolev--Rodriguez-Hertz \cite{GRH21} give the following example by de la Llave \cite{dlL92} that implies we cannot expect any better regularity.

\begin{example}
\label{ex:}
    Let $\alpha\in C^{r}(\sS),\ r\geq 1$ be such that $F(x,y)=(dx,dy+\alpha(x))$ preserves Lebesgue measure $\leb^2$. Then the conjugacy between $(\TT^2,E_d)$ and $(\TT^2,F)$ has the form $H(x,y)=(x,y+\beta(x))$ where $\beta$ is a Weierstrass function explicity expressed as the series
    \[
    \beta(x)=\frac 1d \sum_{k\geq0}\frac1{d^k}\alpha(d^kx)
    \]
    (see \cite[Theorem 6.4]{dlL92} and \cite[Example 7.2]{GRH21}). Gogolev and Rodriguez-Hertz shows that, in this case, $\beta\in C^{x|\log x|}$.
\end{example}

Now we address the number of conjugacies that satisfy Theorem \ref{thm:torusConj}. As in {Corollary \ref{cor:}}, the number of conjugacies is equal to twice the degree of the map.}

{\begin{theorem}\label{thm:Conj}
    Given a skew product $F$ with $\deg f$ and $\deg g_x$ both constant, there are $2(\deg f )(\deg g_x)$ conjugacies that send {the equilibrium state} $\mu$ to Lebesgue $\leb^2$. Moreover, each of these are rotations (orientation preserving or reversing) of each other.
\end{theorem}}

\begin{proof}
{Fist, note that we can define a Markov structure on $\TT^2$ using \cite{FJ79}.} Given a Markov partition for $(\TT^2, F)$ with $\deg(F)=d$, there are $d!$ permutations of the symbols in the partition. For a skew product, a conjugacy must respect the corresponding partitions of the base and fibers. Since we have assumed that $\deg f$ and $\deg g_x$ are constant, then by Corollary \ref{cor:}, there are $d=(\deg f)(\deg g_x)$ (o.p.) conjugacies that send $\mu$ to $\leb^2$.
\end{proof}

\section{Extension to conjugacies on $\TT^n$}
\label{sec:ntorus}

In this section, we note that the results from Section \ref{sec:2torus} can be extended to $(\TT^2,E_d)$. 

\begin{theorem}\label{thm:ntorus}
    For each $n\geq 2$, let $E_d\colon \TT^n\to\TT^n$ be the model map. Let $\varphi\colon\TT^n\to\RR$ be a H\"older continuous function and $\mu$ its corresponding equilibrium state. Then there exists a conjugacy $H\colon\TT^n\to\TT^n$ to a Lebesgue-invariant $C^{1+\alpha}$ expanding skew product $F$ that transports the equilibrium state $\mu$ to Lebesgue measure.    
\end{theorem}

\begin{proof}    
    This is just an induction step on Theorems \ref{thm:Conj} and \ref{thm:det}. Fix some $n\in\NN$. As in the two dimensional case, we apply McMullen to the first coordinate. Let $\hat{\phi}\colon \sS\to\sS$ be such that $f\circ\hat{\phi}=\hat{\phi}\circ E_{d}$ and {$\leb=\hat{\mu}\circ \hat{\phi}$} where $\hat{\mu}=\mu\circ\pi_X^{-1}$.

    \begin{figure}[h!]
        \centering
        \begin{tikzcd}
            \TT_x \arrow[r, "G_x"] \arrow[d,"\phi_x"']
            & \TT_{f(x)} \arrow[d, "\phi_{f(x)}"]\\
            \TT^{n-1} \arrow[r, "E_d"']
            & \TT^{n-1}
        \end{tikzcd}
    
        \caption{Commutative diagram between $(\TT_x,G_x)$ and $(\TT^{n-1},E_d)$.}
        \label{fig:TnFibers}
    \end{figure}
    
    For any $x\in\sS$, let $\TT_x=\{x\}\times\TT^{n-1}$. Assume that there is a family of homeomorphisms $\{\phi_x\colon \TT_x\to\TT^{n-1}\}$ to maps $G_x\colon\TT_x\to\TT_{f(x)}$ that satisfy the following commutative diagram in Figure \ref{fig:TnFibers}. The rest of the proof is analogous to those of Theorems \ref{thm:Conj} and \ref{thm:det}.
\end{proof}

{
It is worth mentioning that de la Llave \cite{dlL92} gives a more general version of Example \ref{ex:} that holds in any dimension and initial regularity of the conjugacy. It shows we cannot hope for any higher regularity of the conjugacy in dimension $n\geq 4$. That is, there exists Anosov diffeomorphisms $f$ and $g$ on $\TT^n$ and $h\in C^k(\TT^n)$ such that $f\circ h=h\circ g$ such that $h\not\in C^{k+1}$.
}

\bibliography{main}
\bibliographystyle{acm}

\end{document}